\documentclass[leqno]{amsart}%
\usepackage{amsmath,amssymb,amscd,latexsym, amsfonts,color}
\usepackage{amsmath}
\usepackage{amsfonts}
\usepackage{amssymb}
\usepackage{graphicx}%
\newtheorem{theorem}{Theorem}[section]
\newtheorem{lemma}[theorem]{Lemma}
\newtheorem{proposition}[theorem]{Proposition}

\newtheorem{remark}[theorem]{Remark}
\newtheorem{example}[theorem]{Example}
\numberwithin{equation}{section}

\begin{document}
\title[NLS equation with singular electromagnetic potential]{Multiple solutions to nonlinear Schr\"{o}dinger equations with singular
electromagnetic potential}

\author{M\'{o}nica Clapp}
\address{Instituto de Matem\'{a}ticas, Universidad Nacional Aut\'{o}noma de
M\'{e}xi\-co, Circuito Exterior, C.U., 04510 M\'{e}xico D.F., Mexico}
\email{mclapp@matem.unam.mx}
\author{Andrzej Szulkin}
\address{Department of Mathematics, Stockholm University, 106 91 Stockholm, Sweden}
\email{andrzejs@math.su.se}
\thanks{The research of M\'{o}nica Clapp was partially supported by CONACYT grant
129847 and UNAM-DGAPA-PAPIIT grant IN106612 (Mexico).}
\thanks{\indent The research of Andrzej Szulkin was partially supported by a grant
from Swedish Research Council.}
\date{}
\maketitle

\begin{center}
\emph{To Kazik G\c{e}ba on the occasion of his birthday, with friendship and
great esteem.}
\end{center}

\begin{abstract}
We consider the semilinear electromagnetic Schr\"{o}dinger equation%
\[
(-i\nabla+\mathcal{A}(x))^{2}u+V(x)u=|u|^{2^{\ast}-2}u,\text{\qquad}u\in
D_{\mathcal{A},0}^{1,2}(\Omega,\mathbb{C}),
\]
where $\Omega=(\mathbb{R}^{m}\smallsetminus\{0\})\times\mathbb{R}^{N-m}$ with
$2\leq m\leq N,$ $N\geq3$, $2^{\ast}:=\frac{2N}{N-2}$ is the critical Sobolev
exponent, $V$ is a Hardy term and $\mathcal{A}$ is a singular magnetic
potential of a particular form which includes the Aharonov-Bohm potentials.
Under some symmetry assumptions on $\mathcal{A}$ we obtain multiplicity of
solutions satisfying certain symmetry properties.

\medskip


\medskip


\end{abstract}

\baselineskip15pt

\section{Introduction}

\label{intro}

Our purpose is to extend some results contained in \cite{at} and \cite{cs}. We
consider the semilinear electromagnetic Schr\"{o}dinger problem
\begin{equation}
\left\{
\begin{array}
[c]{l}%
(-i\nabla+\mathcal{A}(x))^{2}u+V(x)u=|u|^{2^{\ast}-2}u,\\
u\in D_{\mathcal{A},0}^{1,2}(\Omega,\mathbb{C}),
\end{array}
\right.  \label{i1}%
\end{equation}
where $\Omega$ is a domain in $\mathbb{R}^{N},$ $N\geq3$, $2^{\ast}:=\frac
{2N}{N-2}$ is the critical Sobolev exponent, $(-i\nabla+\mathcal{A})^{2}+V$ is
the Hamiltonian for a (non-relativistic) charged particle in an
electromagnetic field, $\mathcal{A}$ is the magnetic (or vector) potential and
$V$ the electric (or scalar) potential for this field. The space
$D_{\mathcal{A},0}^{1,2}(\Omega,\mathbb{C})$ is the closure of $C_{c}^{\infty
}(\Omega,\mathbb{C})$ with respect to the norm
\[
\Vert u\Vert_{\mathcal{A}}:=\left(  \int_{\Omega}|\nabla_{\mathcal{A}}%
u|^{2}\right)  ^{1/2},
\]
where $\nabla_{\mathcal{A}}u:=\nabla u+i\mathcal{A}u.$ If $\Omega
=\mathbb{R}^{N}$, we suppress the index $0$ and write $D_{\mathcal{A}}%
^{1,2}(\mathbb{R}^{N},\mathbb{C})$.

Recall that if $\mathcal{A}=(\mathcal{A}_{1},\ldots,\mathcal{A}_{N})$ is a
vector field defined in some open subset $\Omega$ of $\mathbb{R}^{N}$, then
curl$\mathcal{A}$ is the $N\times N$ skew-symmetric matrix $\mathcal{B}$ with
entries $\mathcal{B}_{jk}=\partial_{j}\mathcal{A}_{k}-\partial_{k}%
\mathcal{A}_{j}$, or in the language of differential forms, if $\mathcal{A}$
is the 1-form $\mathcal{A}_{1}dx_{1}+\cdots+\mathcal{A}_{N}dx_{N}$, then
$\mathcal{B}=d\mathcal{A}$. If $\Omega$ is simply connected and
curl$\mathcal{A}=0$ then $\mathcal{A}=\nabla\varphi$ for some $\varphi$.

In what follows $\mathcal{A}$ will be singular and $V$ will be a Hardy-type potential.

As usual, we denote by $O(m)$ the group of linear isometries of $\mathbb{R}%
^{m}$ and by $SO(m)$ the subgroup of those whose determinant is $1.$ We
identify $SO(2)$ with the group of unit complex numbers $\mathbb{S}^{1}$
acting by multiplication on $\mathbb{C\equiv R}^{2}.$

We start by considering a problem with a point singularity:
\begin{equation}
\left\{
\begin{array}
[c]{l}%
(-i\nabla+\mathcal{A})^{2}u-\frac{\sigma}{|x|^{2}}u=|u|^{2^{\ast}-2}u,\\
u\in D_{\mathcal{A},0}^{1,2}(\Omega,\mathbb{C}),
\end{array}
\right.  \label{point}%
\end{equation}
where $\Omega:=\mathbb{R}^{N}\smallsetminus\{0\}$ and $\mathcal{A}$ is of the
form
\begin{equation}
\mathcal{A}(x)=\frac{\widetilde{\mathcal{A}}(\omega)}{\left\vert x\right\vert
},\qquad\omega:=\frac{x}{\left\vert x\right\vert }\text{\quad and\quad
}\widetilde{\mathcal{A}}\in L^{\infty}(\mathbb{S}^{N-1},\mathbb{R}^{N}).
\label{i2}%
\end{equation}
We fix a closed subgroup $G$ of $O(N)$ and a continuous homomorphism of groups
$\tau:G\rightarrow\mathbb{S}^{1},$ and denote by $Gx:=\{gx:g\in G\}$ the
$G$-orbit of a point $x\in\mathbb{R}^{N}$ and by $G_{x}:=\{g\in G:gx=x\} $ its
isotropy group. We write $\#Gx$ for the cardinality of $Gx$ and $\ker\tau$ for
the kernel of $\tau.$ We assume that $\mathcal{A}$ is $G$-equivariant, that
is,
\begin{equation}
\mathcal{A}(gx)=g\mathcal{A}(x)\text{\qquad for all }x\in\Omega,\text{ }g\in
G. \label{i3}%
\end{equation}
We prove the following result.

\begin{theorem}
\label{thmpoint}Assume that $\#Gx=\infty$ for every $x\in\mathbb{R}%
^{N}\smallsetminus\{0\}$ and that there exists $x^{\ast}\in\mathbb{R}^{N}$
such that $G_{x^{\ast}}\subset\ker\tau.$ Then, if $\mathcal{A}$ satisfies
\emph{(\ref{i2})}, \emph{(\ref{i3})} and $\sigma<\left(  \frac{N-2}{2}\right)
^{2}$, there exists a nontrivial solution $u$ of problem \emph{(\ref{point})}
such that
\begin{equation}
u(gx)=\tau(g)u(x)\text{\qquad}\forall g\in G,\text{ }\forall x\in\Omega.
\label{eq:point}%
\end{equation}
\end{theorem}

Recall that $\left(  \frac{N-2}{2}\right)  ^{2}$ is the Hardy constant.
Theorem \ref{thmpoint} is a special case of a result for more general singular
vector potentials which we state later (cf. Theorem \ref{existence}).

A function $u$ satisfying (\ref{eq:point}) will be called $\tau$%
-\emph{equivariant}. If $\tau\equiv1$ is the trivial homomorphism then
(\ref{eq:point}) simply says that $u$ is $G$-\emph{invariant}, i.e.%
\[
u(gx)=u(x)\text{\qquad}\forall g\in G,\text{ }\forall x\in\Omega.
\]
Condition (\ref{eq:point}) implies that the absolute value $\left\vert
u\right\vert $ of $u$ is $G$-invariant, whereas the phase of $u(gx)$ is that
of $u(x)$ multiplied by $\tau(g).$

The special case where $G=\mathbb{S}^{1}\times O(N-2),$ $N\geq4,$ and
$\tau=\tau_{n}$ with $\tau_{n}(e^{i\vartheta},h):=e^{in\vartheta}$,
$n\in\mathbb{Z}$, was considered by Abatangelo and Terracini in \cite{at}.
Note that, if $u_{n}$ satisfies (\ref{eq:point}) for this data and
$u_{n}(y,z)\neq0$ for some $y\in\mathbb{R}^{2}\smallsetminus\{0\}\equiv
\mathbb{C}\smallsetminus\{0\},$ $z\in\mathbb{R}^{N-2}$, then%
\[
u_{n}(e^{i\vartheta}y,z)=e^{in\vartheta}u_{n}(y,z)\neq0\text{\quad for all
\ }\vartheta\in\lbrack0,2\pi),
\]
i.e. $u_{n}$ restricted to the circle $\mathbb{S}^{1}(y,z):=\{(e^{i\vartheta
}y,z):\vartheta\in\lbrack0,2\pi)\}$ has
degree $n$ (as a mapping into $\mathbb{C}\smallsetminus\{0\}$). Therefore, if all $u_n$ are nontrivial a.e., then $u_{m}\neq u_{n}$ whenever $m\ne n$. So, in this case, Theorem \ref{thmpoint} gives
us infinitely many distinct solutions $u_{n},$ $n\in\mathbb{Z}$, having the
symmetries we just described.

For Aharonov-Bohm potentials we get a stronger result. Again, we write a point
in $\mathbb{R}^{N}$ as $(y,z)\in\mathbb{R}^{2}\times\mathbb{R}^{N-2}$ and
consider the problem
\begin{equation}
\left\{
\begin{array}
[c]{l}%
(-i\nabla+s\mathbb{A})^{2}u-\frac{\sigma}{\left\vert y\right\vert ^{2}%
}u=|u|^{2^{\ast}-2}u,\\
u\in D_{s\mathbb{A},0}^{1,2}(\Omega,\mathbb{C}),
\end{array}
\right.  \tag{$\wp_s$}\label{ab}%
\end{equation}
where $\Omega:=(\mathbb{R}^{2}\smallsetminus\{0\})\times\mathbb{R}^{N-2}%
\equiv(\mathbb{C}\smallsetminus\{0\})\times\mathbb{R}^{N-2},$ $N\geq3,$
$s,\sigma\in\mathbb{R},$ and $\mathbb{A}:\Omega\rightarrow\mathbb{R}^{N}$ is
the Aharonov-Bohm magnetic potential%
\begin{equation}
\mathbb{A}(y_{1},y_{2},z):=\frac{1}{\left\vert y\right\vert ^{2}}(-y_{2}%
,y_{1},0),\text{\qquad}y=(y_{1},y_{2}). \label{AB}%
\end{equation}
We assume that $\sigma<\kappa_{s}$, where $\kappa_{s}$ is the best constant
for the Hardy inequality%
\begin{equation}
\kappa_{s}\int_{\Omega}\frac{\left\vert \varphi\right\vert ^{2}}{\left\vert
y\right\vert ^{2}}\leq\int_{\Omega}\left\vert \nabla\varphi+is\mathbb{A}%
\varphi\right\vert ^{2}\text{\quad}\forall\varphi\in C_{c}^{\infty}%
(\Omega,\mathbb{C}). \label{i4}%
\end{equation}
We prove the following result.

\begin{theorem}
\label{thmunbdd}If $\sigma<\kappa_{s}$, then for every $n\in\mathbb{Z}$ there
exists a nontrivial solution $u_{n}$ of problem (\ref{ab}) with the following properties:

\begin{enumerate}
\item[(a)] $u_{0}=|u_{0}|$ and $|u_{n}|>0$ in $\Omega$,

\item[(b)] $|u_{n}| \neq|u_{m}|$ \ if $|s+m| \neq|s+n|,$

\item[(c)] $u_{n}(e^{i\vartheta}y,z)=e^{in\vartheta}u_{n}(y,z)$ for all
$e^{i\vartheta}\in\mathbb{S}^{1},$ $(y,z)\in\Omega,$

\item[(d)] $u_{n}(y,hz)=u_{n}(y,z)$ for all $h\in O(N-2),$ $(y,z)\in\Omega$,

\item[(e)] $\left\Vert u_{m}\right\Vert _{s\mathbb{A},\sigma}<\left\Vert
u_{n}\right\Vert _{s\mathbb{A},\sigma}$ if $\left\vert s+m\right\vert
<\left\vert s+n\right\vert .$
\end{enumerate}
\end{theorem}

For $N=3$ and $\sigma=0$ this result was proved in \cite{cs}. Properties (c)
and (d) can be expressed by saying that $u_{n}$ satisfies (\ref{eq:point}) for
$G:=\mathbb{S}^{1}\times O(N-2)$ and $\tau=\tau_{n},$ as defined above.

Existence of $\tau$-equivariant solutions can be obtained by constrained
minimization in some appropriate function space, see e.g.\ \cite{cc}. This is
how we proceed to prove Theorem \ref{thmpoint}. For Theorem \ref{thmunbdd} we
give a different argument: we obtain an $\mathbb{S}^{1}$-invariant solution by
minimization and then use the gauge invariance of the Aharonov-Bohm potential
to produce the other solutions. This allows us to establish not only existence
and the symmetry properties, but also all other properties stated in Theorem
\ref{thmunbdd}.

Our Theorems \ref{thmpoint} and \ref{thmunbdd} extend Theorem 1.1 in
\cite{at}. While in \cite{at} it is assumed that $|\sigma|$ is sufficiently
large and $\sigma<0$, we allow all $\sigma<\left(  \frac{N-2}{2}\right)  ^{2}$
and $\sigma<\kappa_{s}$ respectively. We also allow more general symmetries
and our arguments are considerably simpler. On the other hand, in \cite{at} it
was shown that, for $|\sigma|$ and $n$ large enough, there also exists a
solution which satisfies (\ref{eq:point}) with $\tau=\tau_{n}$ for all
$g\in\mathbb{Z}/n\times O(N-2)$ but not for all $g\in\mathbb{S}^{1}\times
O(N-2).$ This does not follow from our results.

Note that the Hardy term in (\ref{point}) is $SO(N)$-invariant. A natural
question arises whether there exist meaningful $SO(N)$-equivariant magnetic
potentials. The answer is contained in the next result.

\begin{theorem}
\label{rad} Let $\Omega\subset\mathbb{R}^{N}$ be a radially symmetric domain
and $\mathcal{A}\in L^{2}_{loc}(\Omega,\mathbb{R}^{N})$, $N\geq3.$

\begin{enumerate}
\item[(i)] If $\mathcal{A}(gx)=g\mathcal{A}(x)$ for all $g\in SO(N)$,
$x\in\Omega$, then $\mathcal{A}(x)=\nu(|x|)x$ for some real-valued function
$\nu$. In particular, $\mathcal{A}(x)=\nabla\varphi(x)$ for some
$SO(N)$-invariant $\varphi\in W^{1,2}_{loc}(\Omega,\mathbb{R})$.

\item[(ii)] If $\mathcal{B}=$ \emph{curl}$\mathcal{A}$ (in the sense of
distributions) for some $\mathcal{A}\in L_{loc}^{2}(\Omega,\mathbb{R}^{N})$
and $\mathcal{B}(gx)=g\mathcal{B}(x)g^{-1}$ for all $g\in SO(N)$, $x\in\Omega
$, then $\mathcal{B}=0$.
\end{enumerate}
\end{theorem}

A consequence of this theorem is that such $\mathcal{A}$ can be gauged away by
an $SO(N)$-invariant gauge transformation, so the problem may be reduced to a
non-magnetic one. More precisely, there is a one-to-one correspondence between
radially symmetric solutions to (\ref{i1}) with $\mathcal{A}$ as above and to
(\ref{i1}) with $\mathcal{A}=0$. We also give an example showing that radial
solutions to (\ref{i1}) may exist even if $\mathcal{A}$ is not $SO(N)$-equivariant.

This paper is organized as follows: Section \ref{prel} contains preliminary
material. In section \ref{sec:existence} we state and prove an extension of
Theorem \ref{thmpoint}\ to more general singular magnetic potentials. Section
\ref{ahar} is devoted to the proof of Theorem \ref{thmunbdd}, and Theorem
\ref{rad} is proved in section \ref{sec:rad}.

\section{Preliminaries}

\label{prel}Let $2\leq m\leq N.$ A point in $\mathbb{R}^{N}$ will be written
as $(y,z)\in\mathbb{R}^{m}\times\mathbb{R}^{N-m}.$ In the following we assume
that
\[
\Omega=(\mathbb{R}^{m}\smallsetminus\{0\})\times\mathbb{R}^{N-m}
\]
and $\mathcal{A}$ is of the form%
\begin{equation}
\mathcal{A}(y,z)=\frac{\widetilde{\mathcal{A}}(\omega,z)}{|y|}\qquad\text{with
\ }\omega:=\frac{y}{\left\vert y\right\vert }\text{\quad and\quad}%
\widetilde{\mathcal{A}}\in L^{\infty}(\mathbb{S}^{m-1}\times\mathbb{R}%
^{N-m},\mathbb{R}^{N}), \label{A}%
\end{equation}
where $\mathbb{S}^{m-1}$ is the unit sphere in $\mathbb{R}^{m}$, centered at
the origin.\ Note that $\mathcal{A}$ in \eqref{i2} and $\mathbb{A}$ in
\eqref{AB} have this form, with $m=N$ and $m=2$ respectively.

\subsection{Basic inequalities}

If $u\in D_{\mathcal{A},0}^{1,2}(\Omega,\mathbb{C})$ then its absolute value
$\left\vert u\right\vert \in D_{0}^{1,2}(\Omega,\mathbb{R})$ and%
\begin{equation}
\left\vert \nabla\left\vert u\right\vert (x)\right\vert \leq\left\vert
\nabla_{\mathcal{A}}u(x)\right\vert \text{\qquad for a.e. }x\in\mathbb{R}^{N}.
\label{di}%
\end{equation}
This is called the \emph{diamagnetic inequality} \cite{ll}. Together with the
Sobolev inequality it yields%
\begin{equation}
S\left(  \int_{\Omega}\left\vert u\right\vert ^{2^{\ast}}\right)  ^{2/2^{\ast
}}\leq\int_{\Omega}\left\vert \nabla_{\mathcal{A}}u\right\vert ^{2},
\label{sob}%
\end{equation}
where $S$ is the best Sobolev constant.

The generalized Hardy inequality \cite{bt} says that
\begin{equation}
\left(  \frac{m-2}{2}\right)  ^{2}\int_{\mathbb{R}^{N}}\frac{\left\vert
v\right\vert ^{2}}{\left\vert y\right\vert ^{2}}\leq\int_{\mathbb{R}^{N}%
}\left\vert \nabla v\right\vert ^{2}\text{\qquad}\forall v\in D_{0}%
^{1,2}(\mathbb{R}^{N},\mathbb{R}). \label{ghardy}%
\end{equation}
The constant $\left(  \frac{m-2}{2}\right)  ^{2}$\ is optimal \cite{ssw}.
Combining this inequality with the diamagnetic inequality we obtain the
\emph{magnetic Hardy inequality}%
\begin{equation}
\left(  \frac{m-2}{2}\right)  ^{2}\int_{\Omega}\frac{\left\vert u\right\vert
^{2}}{\left\vert y\right\vert ^{2}}\leq\int_{\Omega}\left\vert \nabla
_{\mathcal{A}}u\right\vert ^{2}\text{\qquad}\forall u\in D_{\mathcal{A}%
,0}^{1,2}(\Omega,\mathbb{C}). \label{maghardy}%
\end{equation}
A consequence is that%
\[
\Vert u\Vert_{\mathcal{A},\sigma}:=\left(  \int_{\Omega}|\nabla_{\mathcal{A}%
}u|^{2}-\frac{\sigma}{\left\vert y\right\vert ^{2}}|u|^{2}\right)  ^{1/2}%
\]
are equivalent norms in $D_{\mathcal{A},0}^{1,2}(\Omega,\mathbb{C})$ for all
$\sigma<\left(  \frac{m-2}{2}\right)  ^{2}.$

If $m=2$ the inequality (\ref{maghardy}) becomes trivial. Yet for
Aharonov-Bohm potentials there are nontrivial Hardy inequalities. Laptev and
Weidl showed in \cite{lw}\ that $s\mathbb{A}$ in $\mathbb{R}^{2}$ satisfies
\[
\min_{k\in\mathbb{Z}}\left\vert k+s\right\vert ^{2}\int_{\mathbb{R}^{2}}%
\frac{\left\vert \varphi\right\vert ^{2}}{\left\vert y\right\vert ^{2}}%
\leq\int_{\mathbb{R}^{2}}\left\vert \nabla\varphi+is\mathbb{A}\varphi
\right\vert ^{2}\text{\qquad}\forall\varphi\in C_{c}^{\infty}(\mathbb{R}%
^{2}\smallsetminus\{0\},\mathbb{C}).
\]
Integrating this inequality with respect to $z$ we obtain the \emph{Hardy
inequality for Aharonov-Bohm potentials}
\begin{equation}
\min_{k\in\mathbb{Z}}\left\vert k+s\right\vert ^{2}\int_{\mathbb{R}^{N}}%
\frac{\left\vert \varphi\right\vert ^{2}}{\left\vert y\right\vert ^{2}}%
\leq\int_{\mathbb{R}^{N}}\left\vert \nabla_{s\mathbb{A}}\varphi\right\vert
^{2}\text{\qquad}\forall\varphi\in C_{c}^{\infty}(\Omega,\mathbb{C}).
\label{abhardy}%
\end{equation}
It was shown in \cite{at} that $\kappa_{s}:=\min_{k\in\mathbb{Z}}\left\vert
k+s\right\vert ^{2}\in\lbrack0,1/4]$ is the best constant for inequality
(\ref{abhardy}). Again, a consequence is that%
\[
\left\Vert u\right\Vert _{s\mathbb{A},\sigma}:=\left(  \int_{\Omega}\left\vert
\nabla_{s\mathbb{A}}u\right\vert ^{2}-\frac{\sigma}{\left\vert y\right\vert
^{2}}\left\vert u\right\vert ^{2}\right)  ^{1/2}%
\]
are equivalent norms in $D_{s\mathbb{A},0}^{1,2}(\Omega,\mathbb{C})$ for all
$\sigma<\kappa_{s}.$

Note that if $s\in\mathbb{Z}$ then $\kappa_{s}=0$ and the Hardy inequality
(\ref{abhardy}) becomes trivial. We would like to point out that in \cite{cs}
we claimed, \emph{incorrectly}, that $\kappa_{1}>0$. However, we never used
this fact.

\begin{remark}
\emph{ Using (\ref{abhardy}) it was shown in \cite[Section 5.2]{at} that if
$s\not \in \mathbb{Z}$, then $D_{s\mathbb{A},0}^{1,2}(\Omega,\mathbb{C})$ is a
proper subspace of $D^{1,2}(\mathbb{R}^{N},\mathbb{C})$ and the inclusion
mapping is continuous. We point out that, if $s\in\mathbb{Z}\smallsetminus
\{0\}$, the situation is different: $D_{s\mathbb{A},0}^{1,2}(\Omega
,\mathbb{C})$ and $D^{1,2}(\mathbb{R}^{N},\mathbb{C})$ are isometric but none
of them is contained in the other. To see this, let $\theta=\theta(y,z)$ be
the polar angle of $y$ in \eqref{AB} and note that $\nabla\theta=\mathbb{A}$.
Since $\mathbb{R}^{N}\smallsetminus\Omega$ has capacity zero, $D^{1,2}%
(\mathbb{R}^{N},\mathbb{C})=D_{0}^{1,2}(\Omega,\mathbb{C})$ (cf. Theorem 4.7.2
in \cite{eg} and Theorem 9.2.3 in \cite{ma}). Now it follows easily that
$u\mapsto e^{-is\theta}u$ maps $D^{1,2}(\mathbb{R}^{N},\mathbb{C})$
isometrically onto $D_{s\mathbb{A},0}^{1,2}(\Omega,\mathbb{C}).$ But if
$\varphi\in C_{c}^{\infty}(\mathbb{R}^{N},\mathbb{R})$ is such that
$\varphi(x)=1$ for $\left\vert x\right\vert \leq1$, then $\varphi$ is in
$D^{1,2}(\mathbb{R}^{N},\mathbb{C})$ but not in $D_{s\mathbb{A},0}%
^{1,2}(\Omega,\mathbb{C})$ because $1/\left\vert y\right\vert ^{2}$ is not
locally integrable at $0$ (in fact this shows that $D_{s\mathbb{A},0}%
^{1,2}(\Omega,\mathbb{C})\neq D^{1,2}(\mathbb{R}^{N},\mathbb{C})$ for any
$s\neq0$). On the other hand, $e^{-is\theta}\varphi$ is in $D_{s\mathbb{A}%
,0}^{1,2}(\Omega,\mathbb{C})$ but not in $D^{1,2}(\mathbb{R}^{N},\mathbb{C})$.
}
\end{remark}

\subsection{The symmetries}

Let $G$ be a closed subgroup of $O(N).$ We assume that $\Omega$ is
$G$-invariant (i.e. $gx\in\Omega$ for all $x\in\Omega,$ $g\in G$). Since
$\Omega=(\mathbb{R}^{m}\smallsetminus\{0\})\times\mathbb{R}^{N-m}$, this
implies that $\{0\}\times\mathbb{R}^{N-m}$ is $G$-invariant and, therefore,
$G$ must be a subgroup of $O(m)\times O(N-m).$ We also assume that
$\mathcal{A}$ is $G$-equivariant (i.e. $\mathcal{A}(gx)=g\mathcal{A}(x)$ for
all $x\in\Omega,$ $g\in G$). We fix a continuous homomorphism $\tau
:G\rightarrow\mathbb{S}^{1},$ and look for solutions which are $\tau
$-equivariant, i.e.%
\begin{equation}
u(gx)=\tau(g)u(x)\qquad\forall g\in G,\forall x\in\Omega. \label{tauequi}%
\end{equation}

\begin{example}
[Aharonov-Bohm]\label{exAB} \emph{If }$G:=\mathbb{S}^{1}\times O(N-2),$\emph{
the Aharonov-Bohm potentials }$s\mathbb{A}$\emph{ are }$G$\emph{-equivariant.
For each }$n\in\mathbb{Z}$\emph{ we consider the homomorphism }$\tau
_{n}:G\rightarrow\mathbb{S}^{1}$\emph{ given by}%
\[
\tau_{n}(e^{i\vartheta},h):=e^{in\vartheta}\text{\quad}\emph{\ for\ all\ }h\in
O(N-2).
\]
\emph{Note that }$u$\emph{ is }$\tau_{n}$\emph{-equivariant iff it satisfies
properties (c) and (d) of Theorem \ref{thmunbdd}}.
\end{example}

\begin{example}
[Point singularity]\label{expt} \emph{A magnetic potential of the form}%
\[
\mathcal{A}(x)=\frac{\widetilde{\mathcal{A}}\left(  \omega\right)
}{\left\vert x\right\vert }\text{\qquad\emph{with}\quad}\omega:=\frac
{x}{\left\vert x\right\vert },\text{\quad}\widetilde{\mathcal{A}}\in
L^{\infty}(\mathbb{S}^{N-1},\mathbb{R}^{N}),
\]
\emph{is }$G$\emph{-equivariant iff }$\widetilde{\mathcal{A}}$\emph{ is }%
$G$\emph{-equivariant, i.e. }$\widetilde{\mathcal{A}}(g\omega)=g\widetilde
{\mathcal{A}}(\omega)$\emph{ for all }$\omega\in\mathbb{S}^{N-1},$\emph{
}$g\in G.$\emph{ The case where }$G:=\mathbb{S}^{1}\times O(N-2)$\emph{ and
}$\tau=\tau_{n},$\emph{ as above, was considered in \cite{at}. But }$G$\emph{
and }$\tau$\emph{ may be more general. For example, if }$N=2j+k$\emph{ and
}$\Gamma$\emph{ is a subgroup of }$O(k)$\emph{, we may take }$G:=\mathbb{S}%
^{1}\times\Gamma,$\emph{ with }$\mathbb{S}^{1}$\emph{ acting by multiplication
on }$\mathbb{C}^{j}\equiv\mathbb{R}^{2j},$\emph{ and }$\tau_{n}(e^{i\vartheta
},h):=e^{in\vartheta}$\emph{ for all }$h\in\Gamma.$\emph{ Then }$u$\emph{ is
}$\tau_{n}$\emph{-equivariant iff}
\[
u(e^{i\vartheta}\xi_{1},\ldots,e^{i\vartheta}\xi_{j},h\xi^{\prime
})=e^{in\vartheta}u(\xi_{1},\ldots,\xi_{j},\xi^{\prime})\quad\text{\emph{for
all} }\xi_{1},\ldots,\xi_{j}\in\mathbb{C},\text{ }\xi^{\prime}\in
\mathbb{R}^{k}\emph{.}%
\]
\emph{Or we may take} $G:=\mathbb{S}^{1}\times\cdots\times\mathbb{S}^{1}$
\emph{(with }$j$\emph{ factors) and }$\tau_{n_{1},...,n_{j}}(e^{i\vartheta
_{1}},\ldots,e^{i\vartheta_{j}},h):=e^{in_{1}\vartheta_{1}}\cdots
e^{in_{j}\vartheta_{j}}$\emph{ to obtain solutions satisfying}%
\[
u(e^{i\vartheta_{1}}\xi_{1},\ldots,e^{i\vartheta_{j}}\xi_{j},h\xi^{\prime
})=e^{i\left(  n_{1}\vartheta_{1}+\cdots+n_{j}\vartheta_{j}\right)  }u(\xi
_{1},\ldots,\xi_{j},\xi^{\prime})\quad\text{\emph{for all} }\xi_{1},\ldots
,\xi_{j}\in\mathbb{C},\text{ }\xi^{\prime}\in\mathbb{R}^{k}\emph{.}%
\]

\end{example}

Solutions satisfying (\ref{tauequi}) may be obtained by choosing appropriate
subspaces of $D_{\mathcal{A},0}^{1,2}(\Omega,\mathbb{C}).$ To each $u\in
D_{\mathcal{A},0}^{1,2}(\Omega,\mathbb{C}),$ $g\in G,$ we associate the
function $g\cdot_{\tau}u\in D_{\mathcal{A},0}^{1,2}(\Omega,\mathbb{C})$
defined as follows:%
\[
\left(  g\cdot_{\tau}u\right)  (x):=\tau(g)u(g^{-1}x),\text{\qquad}x\in
\Omega.
\]
Then%
\[
\Vert g\cdot_{\tau}u\Vert_{\mathcal{A},\sigma}=\Vert u\Vert_{\mathcal{A}%
,\sigma}\text{\qquad and\qquad}|g\cdot_{\tau}u|_{2^{\ast}}=|u|_{2^{\ast}},
\]
where $\left\vert \, \cdot\,\right\vert _{p}$ denotes the $L^{p}%
(\Omega,\mathbb{C})$-norm. So this gives isometric $G$-actions on
$D_{\mathcal{A},0}^{1,2}(\Omega,\mathbb{C})$ and $L^{2^{\ast}}(\Omega
,\mathbb{C}).$ We write%
\[
D_{\mathcal{A},0}^{1,2}(\Omega,\mathbb{C})^{\tau}:=\{u\in D_{\mathcal{A}%
,0}^{1,2}(\Omega,\mathbb{C}):u(gx)=\tau(g)u(x)\text{\quad}\forall g\in
G,\text{ }\forall x\in\Omega\}
\]
for the fixed point set of this action, which is a closed subspace of
$D_{\mathcal{A},0}^{1,2}(\Omega,\mathbb{C}).$ If $\tau\equiv1$ we write
\[
D_{\mathcal{A},0}^{1,2}(\Omega,\mathbb{C})^{G}:=\{u\in D_{\mathcal{A},0}%
^{1,2}(\Omega,\mathbb{C}):u(gx)=u(x)\text{\quad}\forall g\in G,\text{ }\forall
x\in\Omega\}
\]
instead of $D_{\mathcal{A},0}^{1,2}(\Omega,\mathbb{C})^{\tau}.$

Some of our solutions will be obtained by constrained minimization. If
\[
\inf_{\substack{u\in D_{\mathcal{A},0}^{1,2}(\Omega,\mathbb{C})^{\tau}%
\\u\neq0}}\frac{\Vert u\Vert_{\mathcal{A},\sigma}^{2}}{|u|_{2^{\ast}}^{2}%
}=\inf\left\{  \Vert u\Vert_{\mathcal{A},\sigma}^{2}:u\in D_{\mathcal{A}%
,0}^{1,2}(\Omega,\mathbb{C})^{\tau},\ |u|_{2^{\ast}}^{2}=1\right\}
\]
is attained at some $\bar{u}$ with $|\bar{u}|_{2^{\ast}}=1$, then it follows
from the Lagrange multiplier rule and the principle of symmetric criticality
\cite{p,wi} that $u=\Vert\bar{u}\Vert_{\mathcal{A},\sigma}^{2/(2^{\ast}%
-2)}\bar{u}$ is a solution of (\ref{i1}) satisfying (\ref{tauequi}). Here it
is important to make sure that $D_{\mathcal{A},0}^{1,2}(\Omega,\mathbb{C}%
)^{\tau}$ is nontrivial. A necessary and sufficient condition for this is
given in the following result.

\begin{proposition}
\label{nontrivial} Let $G$ be a closed subgroup of $O(N)$, $\tau:
G\to\mathbb{S}^{1}$ a continuous homomorphism of groups and $\Omega
\subset\mathbb{R}^{N}$ a $G$-invariant domain. \newline(i) If there exists
$x^{*}\in\Omega$ such that $G_{x^{*}}\subset\ker\tau$, then $D_{\mathcal{A}%
,0}^{1,2}(\Omega,\mathbb{C})^{\tau}$ is infinite dimensional. \newline(ii) If
no $x^{*}$ as above exists, then $D_{\mathcal{A},0}^{1,2}(\Omega
,\mathbb{C})^{\tau} = \{0\}$.
\end{proposition}

\begin{proof}
(i) This result has been shown in \cite{ccs}, see the proof of Theorem 1.1
there. For the reader's convenience we include a (slightly different)
argument. Recall that two groups $H,K \subset G$ are called conjugate in $G$
if $K=g^{-1}Hg$ for some $g\in G$ and denote the conjugacy class of $H$ by
$(H)$. By Theorem 5.14 in Chapter I of \cite{td} there exists a unique
conjugacy class $(P)$, $P\subset G$, such that the set
\[
\Omega_{(P)} := \{x\in\mathbb{R}^{N}: G_{x} \text{ is conjugate to } P \text{
in }G\}
\]
is open and dense in $\Omega$. Moreover, if $x\in\Omega\smallsetminus
\Omega_{(P)}$, then $P$ is conjugate to a (proper) subgroup of $G_{x}$. Let
$x^{*}\in\Omega_{(P)}$, let $N$ be an open $G$-invariant tubular neighborhood
of $G_{x^{*}}$ in $\Omega$ and denote the fiber over $x^{*}$ in $N$ by
$N_{x^{*}}$. Then $\dim N_{x^{*}} = N-\dim G_{x^{*}} > 0$. For each
$\varphi\in C_{c}^{\infty}(N_{x^{*}},\mathbb{C})$ we can now define
\[
\widetilde\varphi(gx) := \tau(g)\varphi(x), \quad x\in N_{x^{*}}, \ g\in G.
\]
Since $G_{x}$ is conjugate to $P$, $G_{x}\subset\ker\tau$ and it follows that
$\widetilde\varphi$ is well defined. Clearly, the space of all $\widetilde
\varphi$ as above is an infinite dimensional subspace of $C_{c}^{\infty
}(\Omega,\mathbb{C})$. \newline(ii) If for each $x\in\Omega$ there exists
$g\in G_{x}$ such that $\tau(g)\ne1$, then $u(x)=u(gx) = \tau(g)u(x)$ and,
hence, $u(x)=0$.
\end{proof}

As an example that (ii) can occur we may take $G=O(N)$ and the determinant of
$g$ as $\tau(g)$. Then for each $x\in\Omega$ there exists a $g_{x}\in O(N)$
with $g_{x}x=x$ and $\tau(g_{x})=-1.$

\subsection{Concentration-compactness}

An essential tool in our proofs will be the concentration-compactness lemma
which we now formulate. Let $\mathcal{M}(\mathbb{R}^{N})$ denote the space of
finite measures in $\mathbb{R}^{N}$ and set
\[
S_{\mathcal{A},\sigma}^{\tau}:=\inf_{\substack{u\in D_{\mathcal{A},0}%
^{1,2}(\Omega,\mathbb{C})^{\tau}\\u\neq0}}\frac{\Vert u\Vert_{\mathcal{A}%
,\sigma}^{2}}{|u|_{2^{\ast}}^{2}}.
\]
If $\tau\equiv1$ we write $S_{\mathcal{A},\sigma}^{G}$ instead of
$S_{\mathcal{A},\sigma}^{\tau}.$

\begin{lemma}
[Concentration-compactness]\label{cc} Let $\sigma<\left(  \frac{m-2}%
{2}\right)  ^{2}$ or $\sigma<\kappa_{s}$ if $\mathcal{A}=s\mathbb{A}$, and let
$(u_{n})$ be a sequence in $D_{\mathcal{A},0}^{1,2}(\Omega,\mathbb{C})^{\tau}$
such that%
\[%
\begin{array}
[c]{rl}%
u_{n}\rightharpoonup u & \text{weakly in }D_{\mathcal{A},0}^{1,2}%
(\Omega,\mathbb{C}),\\
|\nabla_{\mathcal{A}}(u_{n}-u)|^{2}-\frac{\sigma}{\left\vert y\right\vert
^{2}}|u_{n}-u|^{2}\rightharpoonup\mu & \text{weakly in }\mathcal{M}%
(\mathbb{R}^{N}),\\
|u_{n}-u|^{2^{\ast}}\rightharpoonup\nu & \text{weakly in }\mathcal{M}%
(\mathbb{R}^{N}),\\
u_{n}(x)\rightarrow u(x) & \text{a.e. in }\mathbb{R}^{N}.
\end{array}
\]
Set
\begin{gather*}
\mu_{\infty}:=\lim_{R\rightarrow\infty}\,\limsup_{n\rightarrow\infty}%
\int_{|x|\geq R}\left(  |\nabla_{\mathcal{A}}u_{n}|^{2}-\frac{\sigma
}{\left\vert y\right\vert ^{2}}|u_{n}|^{2}\right) \\
\text{and}\quad\nu_{\infty}:=\lim_{R\rightarrow\infty}\,\limsup_{n\rightarrow
\infty}\int_{|x|\geq R}|u_{n}|^{2^{\ast}}.
\end{gather*}
Then, the following hold:

\begin{enumerate}
\item[(a)] $S_{\mathcal{A},\sigma}^{\tau}\left\Vert \nu\right\Vert
^{2/2^{\ast}}\leq\Vert\mu\Vert$.

\item[(b)] $S_{\mathcal{A},\sigma}^{\tau}\nu_{\infty}^{2/2^{\ast}}\leq
\mu_{\infty}$.

\item[(c)] $\limsup_{n\rightarrow\infty}\Vert u_{n}\Vert_{\mathcal{A},\sigma
}^{2}=\Vert u\Vert_{{\mathcal{A},\sigma}}^{2}+\left\Vert \mu\right\Vert
+\mu_{\infty}$.

\item[(d)] $\limsup_{n\rightarrow\infty}|u_{n}|_{2^{\ast}}^{2^{\ast}%
}=|u|_{2^{\ast}}^{2^{\ast}}+\left\Vert \nu\right\Vert +\nu_{\infty}$.

\item[(e)] If $u=0$ and $S_{\mathcal{A},\sigma}^{\tau}\left\Vert
\nu\right\Vert ^{2/2^{\ast}}=\Vert\mu\Vert$, then $\mu$ and $\nu$ are
concentrated at a single finite $G$-orbit of $\mathbb{R}^{N}$ or are zero.
\end{enumerate}
\end{lemma}

The proof uses the same argument as in \cite{cs} which, in turn, is a
modification of that in \cite{as}, see also \cite{wa} where the
concentration-compactness lemma for $G$-invariant functions was introduced.
There are, however, two facts which were obvious in \cite{cs} but need to be
verified here. If $\sigma>0$, then it is not clear that $\mu$ is a positive
measure and that $\mu_{\infty}$ is well defined and nonnegative. We show this
in the next lemma.

\begin{lemma}
\label{positivity} $\mu_{\infty}\ge0$ and $\mu$ is a positive measure.
\end{lemma}

\begin{proof}
The arguments below are adaptations of those in the Appendix of \cite{sw} (for
$\mu_{\infty}$) and Lemma 2.3 in \cite{csw} (for $\mu$). Since the conclusions
are obvious if $\sigma\leq0$, we assume that $\sigma>0$.

First we show that $\mu$ is a positive measure. Let $\varphi\in C_{c}^{\infty
}(\mathbb{R}^{N},\mathbb{R})$ be $G$-invariant, $\varphi\geq0$, and put
$\varphi_{\varepsilon}:=(\sqrt{\varphi+\varepsilon^{2}}-\varepsilon)^{2}$
($\varepsilon>0$) and $v_{n}:=u_{n}-u$. Using the Hardy inequality
(\ref{ghardy}) or (\ref{abhardy}) and the fact that $\sqrt{\varphi
_{\varepsilon}}$ is differentiable we obtain
\begin{align}
\label{meas}0  &  \leq\int_{\mathbb{R}^{N}}\left(  |\nabla_{\mathcal{A}}%
(\sqrt{\varphi_{\varepsilon}}v_{n})|^{2}-\frac{\sigma}{\left\vert y\right\vert
^{2}}|\sqrt{\varphi_{\varepsilon}}v_{n}|^{2}\right) \\
&  =\int_{\mathbb{R}^{N}}\varphi_{\varepsilon}\left(  |\nabla_{\mathcal{A}%
}v_{n}|^{2}-\frac{\sigma}{\left\vert y\right\vert ^{2}}|v_{n}|^{2}\right)
+2\operatorname{Re}\int_{\mathbb{R}^{N}}\sqrt{\varphi_{\varepsilon}}\bar
{v}_{n}\nabla_{\mathcal{A}}v_{n}\cdot\nabla(\sqrt{\varphi_{\varepsilon}%
})\nonumber\\
&  \qquad+\int_{\mathbb{R}^{N}}|\nabla(\sqrt{\varphi_{\varepsilon}}%
)|^{2}|v_{n}|^{2}.\nonumber
\end{align}
Since $v_{n}\rightarrow0$ in $L_{loc}^{2}(\mathbb{R}^{N},\mathbb{C})$, the
second and the third term on the right-hand side above go to $0$ as
$n\rightarrow\infty$. So
\[
\langle\mu, \varphi_{\varepsilon}\rangle= \lim_{n\to\infty} \int
_{\mathbb{R}^{N}}\varphi_{\varepsilon}\left(  |\nabla_{\mathcal{A}}v_{n}%
|^{2}-\frac{\sigma}{\left\vert y\right\vert ^{2}}|v_{n}|^{2}\right)  \geq0
\quad\text{and}\quad\langle\mu, \varphi\rangle\ge0
\]
because $\varphi_{\varepsilon}\to\varphi$ in $L^{\infty}(\mathbb{R}%
^{N},\mathbb{R})$ as $\varepsilon\to0$. This proves that $\mu\geq0$.

Next, we show that $\mu_{\infty}$ is well defined and nonnegative. Let
$\psi_{R}\in C^{\infty}(\mathbb{R}^{N},[0,1])$ be radially symmetric and such
that $\psi_{R}(x)=0$ for $\left\vert x\right\vert \leq R$, $\psi_{R}(x)=1$ for
$\left\vert x\right\vert \geq R+1$. Then
\begin{align}
&  \limsup_{n\rightarrow\infty}\Vert u_{n}\Vert_{\mathcal{A},\sigma}^{2}=\Vert
u\Vert_{\mathcal{A},\sigma}^{2}+\limsup_{n\rightarrow\infty}\Vert v_{n}%
\Vert_{\mathcal{A},\sigma}^{2}\label{infty}\\
&  =\Vert u\Vert_{\mathcal{A},\sigma}^{2}+\lim_{R\rightarrow\infty}%
\limsup_{n\rightarrow\infty}\left(  \int_{\mathbb{R}^{N}}(1-\psi_{R}%
^{2})\left(  |\nabla_{\mathcal{A}}v_{n}|^{2}-\frac{\sigma}{\left\vert
y\right\vert ^{2}}|v_{n}|^{2}\right)  \right. \nonumber\\
&  \qquad+\left.  \int_{\mathbb{R}^{N}}\psi_{R}^{2}\left(  |\nabla
_{\mathcal{A}}v_{n}|^{2}-\frac{\sigma}{\left\vert y\right\vert ^{2}}%
|v_{n}|^{2}\right)  \right) \nonumber\\
&  =\Vert u\Vert_{\mathcal{A},\sigma}^{2}+\lim_{R\rightarrow\infty}\left(
\int_{\mathbb{R}^{N}}(1-\psi_{R}^{2})\,d\mu+\limsup_{n\rightarrow\infty}%
\int_{\mathbb{R}^{N}}\psi_{R}^{2}\left(  |\nabla_{\mathcal{A}}v_{n}|^{2}%
-\frac{\sigma}{\left\vert y\right\vert ^{2}}|v_{n}|^{2}\right)  \right)
\nonumber\\
&  =\Vert u\Vert_{\mathcal{A},\sigma}^{2}+\Vert\mu\Vert+\lim_{R\rightarrow
\infty}\limsup_{n\rightarrow\infty}\int_{\mathbb{R}^{N}}\psi_{R}^{2}\left(
|\nabla_{\mathcal{A}}v_{n}|^{2}-\frac{\sigma}{\left\vert y\right\vert ^{2}%
}|v_{n}|^{2}\right) \nonumber\\
&  =\Vert u\Vert_{\mathcal{A},\sigma}^{2}+\Vert\mu\Vert+\lim_{R\rightarrow
\infty}\limsup_{n\rightarrow\infty}\int_{\mathbb{R}^{N}}\psi_{R}^{2}\left(
|\nabla_{\mathcal{A}}u_{n}|^{2}-\frac{\sigma}{\left\vert y\right\vert ^{2}%
}|u_{n}|^{2}\right)  ,\nonumber
\end{align}
where the last equality follows from the fact that
\[
\lim_{R\rightarrow\infty}\int_{\mathbb{R}^{N}}\psi_{R}^{2}\left(
|\nabla_{\mathcal{A}}u|^{2}-\frac{\sigma}{\left\vert y\right\vert ^{2}}%
|u|^{2}\right)  =0.
\]
Since
\[
\int_{\mathbb{R}^{N}}\psi_{R}^{2}|\nabla_{\mathcal{A}}u_{n}|^{2}\leq
\int_{|x|>R}|\nabla_{\mathcal{A}}u_{n}|^{2}\leq\int_{\mathbb{R}^{N}}\psi
_{R-1}^{2}|\nabla_{\mathcal{A}}u_{n}|^{2}%
\]
and
\[
\int_{\mathbb{R}^{N}}\psi_{R}^{2}\frac{|u_{n}|^{2}}{\left\vert y\right\vert
^{2}}\leq\int_{|x|>R}\frac{|u_{n}|^{2}}{\left\vert y\right\vert ^{2}}\leq
\int_{\mathbb{R}^{N}}\psi_{R-1}^{2}\frac{|u_{n}|^{2}}{\left\vert y\right\vert
^{2}},
\]
we have, setting $I_{R}(u_{n}):=\int_{\mathbb{R}^{N}}(\psi_{R-1}^{2}-\psi
_{R}^{2})\frac{\sigma}{\left\vert y\right\vert ^{2}}|u_{n}|^{2},$
\begin{align}
&  \int_{\mathbb{R}^{N}}\psi_{R}^{2}\left(  |\nabla_{\mathcal{A}}u_{n}%
|^{2}-\frac{\sigma}{\left\vert y\right\vert ^{2}}|u_{n}|^{2}\right)
-I_{R}(u_{n})\leq\int_{|x|>R}\left(  |\nabla_{\mathcal{A}}u_{n}|^{2}%
-\frac{\sigma}{\left\vert y\right\vert ^{2}}|u_{n}|^{2}\right) \label{p10}\\
&  \qquad\leq\int_{\mathbb{R}^{N}}\psi_{R-1}^{2}\left(  |\nabla_{\mathcal{A}%
}u_{n}|^{2}-\frac{\sigma}{\left\vert y\right\vert ^{2}}|u_{n}|^{2}\right)
+I_{R}(u_{n}).\nonumber
\end{align}
Since $u_{n}\rightarrow u$ in $L_{loc}^{2}(\mathbb{R}^{N},\mathbb{C})$,
$\lim_{R\rightarrow\infty}\lim_{n\rightarrow\infty}I_{R}(u_{n})=0$. So
(\ref{p10}) implies that%
\begin{align*}
\mu_{\infty}:=  &  \lim_{R\rightarrow\infty}\limsup_{n\rightarrow\infty}%
\int_{|x|>R}\left(  |\nabla_{\mathcal{A}}u_{n}|^{2}-\frac{\sigma}{\left\vert
y\right\vert ^{2}}|u_{n}|^{2}\right) \\
&  =\lim_{R\rightarrow\infty}\limsup_{n\rightarrow\infty}\int_{\mathbb{R}^{N}%
}\psi_{R}^{2}\left(  |\nabla_{\mathcal{A}}u_{n}|^{2}-\frac{\sigma}{\left\vert
y\right\vert ^{2}}|u_{n}|^{2}\right)  .
\end{align*}
Since the right-hand side above is well defined according to \eqref{infty},
also $\mu_{\infty}$ is well defined. Note that we have also proved that
$\mu_{\infty}$ satisfies (c) of the lemma.

Finally, similarly as in \eqref{meas}, we obtain
\begin{align*}
0  &  \leq\int_{\mathbb{R}^{N}}\left(  |\nabla_{\mathcal{A}}(\psi_{R}%
u_{n})|^{2}-\frac{\sigma}{\left\vert y\right\vert ^{2}}|\psi_{R}u_{n}%
|^{2}\right) \\
&  =\int_{\mathbb{R}^{N}}\psi_{R}^{2}\left(  |\nabla_{\mathcal{A}}u_{n}%
|^{2}-\frac{\sigma}{\left\vert y\right\vert ^{2}}|u_{n}|^{2}\right)
+2\operatorname{Re}\int_{\mathbb{R}^{N}}\psi_{R}\bar{u}_{n}\nabla
_{\mathcal{A}}u_{n}\cdot\nabla\psi_{R} +\int_{\mathbb{R}^{N}}|\nabla\psi
_{R}|^{2}|u_{n}|^{2},
\end{align*}
and by the convergence of $u_{n}$ in $L_{loc}^{2}(\mathbb{R}^{N},\mathbb{C})$,
the second and the third term on the right-hand side above tend to $0$ if one
first lets $n\rightarrow\infty$ and then $R\rightarrow\infty$. We conclude
that $\mu_{\infty}\geq0$.
\end{proof}

\section{Existence of minimizers}

\label{sec:existence}We assume that $\Omega=(\mathbb{R}^{m}\smallsetminus
\{0\})\times\mathbb{R}^{N-m}$ with $2\leq m\leq N$ and $\mathcal{A}$ is of the
form (\ref{A}). We write $(y,z)\in\mathbb{R}^{m}\times\mathbb{R}^{N-m}%
\equiv\mathbb{R}^{N}$ and consider the problem
\begin{equation}
\left\{
\begin{array}
[c]{l}%
(-i\nabla+\mathcal{A})^{2}u-\frac{\sigma}{|y|^{2}}u=|u|^{2^{\ast}-2}u,\\
u\in D_{\mathcal{A},0}^{1,2}(\Omega,\mathbb{C}).
\end{array}
\right.  \label{genprob}%
\end{equation}
Let $m\leq M\leq N$, $\ G$ be a closed subgroup of $O(m)\times O(M-m),$
considered as a subgroup of $O(M)$ in the obvious way, and $\tau
:G\rightarrow\mathbb{S}^{1}$ be a continuous homomorphism. Recall that
$G\xi:=\{g\xi:g\in G\}$ is the $G$-orbit of $\xi$ and $G_{\xi}:=\{g\in
G:g\xi=\xi\}$ its isotropy group. We write $\mathbb{R}^{N}\equiv\mathbb{R}%
^{M}\times\mathbb{R}^{N-M}$\ and assume the following:

\begin{description}
\item[$\left(  H_{1}\right)  $] $\#G\xi=\infty$ for every $\xi\in
\mathbb{R}^{M}\smallsetminus\{0\}.$

\item[$\left(  H_{2}\right)  $] There exists $\xi^{\ast}\in\mathbb{R}^{M}$
such that $G_{\xi^{\ast}}\subset\ker\tau.$

\item[$\left(  H_{3}\right)  $] $\mathcal{A}$ is $G$-equivariant.

\item[$\left(  H_{4}\right)  $] $\mathcal{A}(\xi,\zeta)=\mathcal{A}(\xi,0)$
\ for all $\xi\in(\mathbb{R}^{m}\smallsetminus\{0\})\times\mathbb{R}^{M-m},$
$\zeta\in\mathbb{R}^{N-M}.$
\end{description}

Hypothesis $(H_{2})$ guarantees that $D_{\mathcal{A},0}^{1,2}(\Omega
,\mathbb{C})^{\tau}$ is infinite dimensional, see Proposition \ref{nontrivial}%
. At the end of the section we give some examples of group actions satisfying
our assumptions.

Note that the Hardy potential $V(y,z):=\frac{\sigma}{\left\vert y\right\vert
^{2} }$ is $\left[  O(m)\times O(M-m)\right]  $-invariant and, since $m\leq
M,$ it satisfies $V(\xi,\zeta)=V(\xi,0)$ for all $\xi\in(\mathbb{R}%
^{m}\smallsetminus\{0\})\times\mathbb{R}^{M-m},$ $\zeta\in\mathbb{R}^{N-M}.$

The following result is an extension of Proposition 3.2 in \cite{cs}. The
proof is similar. We give the details for the reader's convenience.

\begin{theorem}
\label{existence}Assume that $(H_{1})$-$(H_{4})$ hold, and let $\sigma<\left(
\frac{m-2}{2}\right)  ^{2}$ or $\sigma<\kappa_{s}$ if $\mathcal{A}%
=s\mathbb{A}$. Then $S_{\mathcal{A},\sigma}^{\tau}$ is attained. Therefore
problem \emph{(\ref{genprob})} has a solution $u\in D_{\mathcal{A},0}%
^{1,2}(\Omega,\mathbb{C})$ which satisfies%
\[
u(g\xi,\zeta)=\tau(g)u(\xi,\zeta)\text{\qquad for all }g\in G.
\]
\end{theorem}

\begin{proof}
Let $(u_{k})$ be a sequence in $D_{\mathcal{A},0}^{1,2}(\Omega,\mathbb{C}%
)^{\tau}$ such that $\left\vert u_{k}\right\vert _{2^{\ast}}=1$ and
$\left\Vert u_{k}\right\Vert _{\mathcal{A},\sigma}^{2}\rightarrow
S_{\mathcal{A},\sigma}^{\tau}.$ For each $k\in\mathbb{N}$ set%
\[
Q_{k}(r):=\sup_{\zeta\in\mathbb{R}^{N-M}}\int_{B_{r}(0,\zeta)}\left\vert
u_{k}\right\vert ^{2^{\ast}}.
\]
Then $Q_{k}(r)\rightarrow0$ as $r\rightarrow0$ and $Q_{k}(r)\rightarrow1$ as
$r\rightarrow\infty$. Hence, there exist $r_{k}\in(0,\infty)$ and $\zeta
_{k}\in\mathbb{R}^{N-M}$ such that
\[
Q_{k}(r_{k}) := \sup_{\zeta\in\mathbb{R}^{N-M}}\int_{B_{r_{k}}(0,\zeta
)}\left\vert u_{k}\right\vert ^{2^{\ast}}=\int_{B_{r_{k}}(0,\zeta_{k}%
)}\left\vert u_{k}\right\vert ^{2^{\ast}}=\frac{1}{2}.
\]
Let $w_{k}(\xi,\zeta):=r_{k}^{(N-2)/2}u_{k}(r_{k}\xi,r_{k}\zeta+\zeta_{k}).$
Then $w_{k}\in D_{\mathcal{A},0}^{1,2}(\Omega,\mathbb{C})^{\tau}$, $\left\vert
w_{k}\right\vert _{2^{\ast}}=1$, $\left\Vert w_{k}\right\Vert _{\mathcal{A}%
,\sigma}^{2}=\left\Vert u_{k}\right\Vert _{\mathcal{A},\sigma}^{2}$ and%
\begin{equation}
\sup_{\zeta\in\mathbb{R}^{N-M}}\int_{B_{1}(0,\zeta)}\left\vert w_{k}%
\right\vert ^{2^{\ast}}=\int_{B_{1}(0,0)}\left\vert w_{k}\right\vert
^{2^{\ast}}=\frac{1}{2}. \label{nonvanish}%
\end{equation}
Passing to a subsequence, we have that $w_{k}\rightharpoonup w$ weakly in
$D_{\mathcal{A},0}^{1,2}(\Omega,\mathbb{C})^{\tau},$ $w_{k}(x)\rightarrow
w(x)\ $a.e. in $\mathbb{R}^{N},\ \ |\nabla_{\mathcal{A}}(w_{k}-w)|^{2}%
-\frac{\sigma}{\left\vert y\right\vert ^{2}}|w_{k}-w|^{2}\rightharpoonup\mu$
\ and $|w_{k}-w|^{2^{\ast}}\rightharpoonup\nu$ weakly in $\mathcal{M}%
(\mathbb{R}^{N}).$ Since%
\[
\lim_{k\rightarrow\infty}\Vert w_{k}\Vert_{\mathcal{A},\sigma}^{2}%
=S_{\mathcal{A},\sigma}^{\tau}\lim_{k\rightarrow\infty}|w_{k}|_{2^{\ast}}%
^{2},
\]
using Lemma \ref{cc} and the definition of $S_{\mathcal{A},\sigma}^{\tau}$ we
obtain%
\begin{align*}
\Vert w\Vert_{\mathcal{A},\sigma}^{2}+\left\Vert \mu\right\Vert +\mu_{\infty}
&  =S_{\mathcal{A},\sigma}^{\tau}\left(  |w|_{2^{\ast}}^{2^{\ast}}+\left\Vert
\nu\right\Vert +\nu_{\infty}\right)  ^{2/2^{\ast}}\\
&  \leq S_{\mathcal{A},\sigma}^{\tau}\left(  |w|_{2^{\ast}}^{2}+\left\Vert
\nu\right\Vert ^{2/2^{\ast}}+\nu_{\infty}^{2/2^{\ast}}\right)  \leq\Vert
w\Vert_{\mathcal{A},\sigma}^{2}+\left\Vert \mu\right\Vert +\mu_{\infty}.
\end{align*}
Hence%
\[
\left(  |w|_{2^{\ast}}^{2^{\ast}}+\left\Vert \nu\right\Vert +\nu_{\infty
}\right)  ^{2/2^{\ast}}=|w|_{2^{\ast}}^{2}+\left\Vert \nu\right\Vert
^{2/2^{\ast}}+\nu_{\infty}^{2/2^{\ast}}.
\]
It follows that one of the quantities $\left\vert w\right\vert _{2^{\ast}%
},\left\Vert \nu\right\Vert ,\nu_{\infty}$ is $1$ and the other two are $0.$
Equality (\ref{nonvanish}) implies that $\nu_{\infty}\neq1,$ hence
$\nu_{\infty}=0.$ Assume $\left\Vert \nu\right\Vert =1.$ Then $w=0$ and
$S_{\mathcal{A},\sigma}^{\tau}\left\Vert \nu\right\Vert ^{2/2^{\ast}}=\Vert
\mu\Vert,$ so $\nu$ is concentrated at the $G$-orbit of a point $(0,\zeta
^{\ast})$ with $\zeta^{\ast}\in\mathbb{R}^{N-M}$ (all other $G$-orbits are
infinite because of $(H_{1})$). But since $\#G(0,\zeta^{\ast})=1$,
(\ref{nonvanish}) yields%
\[
\frac{1}{2}=\int_{B_{1}(0,0)}\left\vert w_{k}\right\vert ^{2^{\ast}}\geq
\lim_{k\rightarrow\infty}\int_{B_{1}(0,\zeta^{\ast})}\left\vert w_{k}%
\right\vert ^{2^{\ast}}=\left\Vert \nu\right\Vert =1.
\]
This is a contradiction. Therefore $\left\Vert \nu\right\Vert =0,$ $\left\vert
w\right\vert _{2^{\ast}}=1$ and $\Vert w\Vert_{\mathcal{A},\sigma}%
^{2}=S_{\mathcal{A},\sigma}^{\tau}$.
\end{proof}

Theorem \ref{thmpoint}\ is just a special case of the previous result with
$m=M=N.$ In this case, assumption $(H_{4})$ is trivially satisfied. Assumption
$(H_{1})$ holds for the actions of $G=\mathbb{S}^{1}$ or $G=\mathbb{S}%
^{1}\times\cdots\times\mathbb{S}^{1}$\ described in Example \ref{expt} if, for
example, $k=0$ or $k\geq2$ and $\Gamma=O(k)$ (but not if $k=1$!)$.$ The
choices of $\tau$ described there will give solutions with the corresponding
symmetry properties (note that $(H_{2})$ holds).

For the Aharonov-Bohm potential $s\mathbb{A},$ taking $m=2$, $M=N,$
$G=\mathbb{S}^{1}\times O(N-2)$ and $\tau_{n}(e^{i\vartheta},g)=e^{in\vartheta
}$ in Theorem \ref{existence}\ we obtain a solution to problem (\ref{ab})
satisfying (c) and (d) of Theorem \ref{thmunbdd}\ when $N\geq4.$ If $N=3$ this
$G$ does not satisfy $(H_{1}).$ However, since $s\mathbb{A}$ does not depend
on the last variable, $(H_{1})$-$(H_{4})$ hold for $m=M=2$ and $G=\mathbb{S}%
^{1},$ so we may apply Theorem \ref{existence} with $\tau_{n}(e^{i\vartheta
})=e^{in\vartheta}$ to obtain a solution which satisfies (c). In order to get
the other statements of Theorem \ref{thmunbdd} we shall proceed in a different way.

\section{The Aharonov-Bohm potential}

\label{ahar}Set $\Omega:=(\mathbb{R}^{2}\smallsetminus\{0\})\times
\mathbb{R}^{N-2}.$ Although the curl of the Aharonov-Bohm potential (\ref{AB})
is zero, $\mathbb{A}$ is not the gradient of a function defined on $\Omega.$
However, it is the gradient of a function which is defined locally. Namely, if
$\theta(y,z)$ denotes the polar angle of $y\in\mathbb{C}\equiv\mathbb{R}^{2}$,
then $\nabla\theta=\mathbb{A}.$ The function $e^{i\theta}$ is well defined
(while $\theta$ is only defined up to an integer multiple of $2\pi$) and a
direct computation shows that if $v_{s+n}$ is a solution of $(\wp{_{s+n}})$,
then $e^{in\theta}v_{s+n}$ is a solution of (\ref{ab}). As we did in
\cite{cs}, we use this fact to construct $u_{n}$ by taking $v_{s+n}$ to be a
real-valued $\mathbb{S}^{1}$-invariant solution of $(\wp{_{s+n}})$ and setting
$u_{n}:=e^{in\theta}v_{s+n}$. These solutions will have the properties
asserted by Theorem \ref{thmunbdd}. Since the proof of Theorem \ref{thmunbdd}
is similar to that of Theorem 1.3 in \cite{cs}, we omit some details and
concentrate on pointing out the differences.

If $\varphi\in C_{c}^{\infty}(\Omega,\mathbb{C})$ is $\mathbb{S}^{1}%
$-invariant, i.e. constant on every circle $\{(y,z_{0}):\left\vert
y\right\vert =r\},$ $r>0,$ $z_{0}\in\mathbb{R}^{N-2}$ then, since
$\mathbb{A}(y,z_{0})$ is tangent to the circle, we have that $\nabla
\varphi(x)\cdot\mathbb{A}(x)=0$ for all $x\in\Omega$. Hence,
\begin{equation}
|\nabla\varphi+is\mathbb{A}\varphi|^{2}=|\nabla\varphi|^{2}+|s\mathbb{A}%
\varphi|^{2} \label{gradradial}%
\end{equation}
and
\begin{align}
\Vert u\Vert_{s\mathbb{A},\sigma}^{2}  &  =\int_{\Omega}\left(  |\nabla
u|^{2}+|s\mathbb{A}|^{2}|u|^{2}-\frac{\sigma}{\left\vert y\right\vert ^{2}%
}|u|^{2}\right) \label{ab2}\\
&  =\int_{\Omega}\left(  |\nabla u|^{2}+\frac{s^{2}-\sigma}{\left\vert
y\right\vert ^{2}}|u|^{2}\right)  \quad\text{\ for all }u\in D_{s\mathbb{A}%
,0}^{1,2}(\Omega,\mathbb{C})^{\mathbb{S}^{1}}.\nonumber
\end{align}
In this case, the minimizer can be chosen to be real-valued. More precisely,
the following holds.

\begin{lemma}
\label{AB1}Let $\sigma<\kappa_{s}$ and $G=\mathbb{S}^{1}\times\Gamma,$ where
$\Gamma$ is a subgroup of $O(N-2)$. Then
\begin{equation}
S_{s\mathbb{A},\sigma}^{G}:=\inf_{\substack{u\in D_{s\mathbb{A},0}%
^{1,2}(\Omega,\mathbb{C})^{G}\\u\neq0}}\frac{\Vert u\Vert_{s\mathbb{A},\sigma
}^{2}}{|u|_{2^{\ast}}^{2}}=\inf_{\substack{u\in D_{s\mathbb{A},0}^{1,2}%
(\Omega,\mathbb{R})^{G}\\u\neq0}}\frac{\Vert u\Vert_{s\mathbb{A},\sigma}^{2}%
}{|u|_{2^{\ast}}^{2}}. \label{SG}%
\end{equation}
Moreover, if $u$ is a minimizer for the left-hand side, then $|u|$ is a
minimizer for the right-hand side.
\end{lemma}

\begin{proof}
Note that $D_{s\mathbb{A},0}^{1,2}(\Omega,\mathbb{C})^{G}\subset
D_{s\mathbb{A},0}^{1,2}(\Omega,\mathbb{C})^{\mathbb{S}^{1}}$. Applying the
diamagnetic inequality and (\ref{ab2}) we obtain%
\begin{align*}
\Vert\left\vert u\right\vert \Vert_{s\mathbb{A},\sigma}^{2}  &  =\int_{\Omega
}\left(  |\nabla\left\vert u\right\vert |^{2}+\frac{s^{2}-\sigma}{\left\vert
y\right\vert ^{2}}|u|^{2}\right) \\
&  \leq\int_{\Omega}\left(  |\nabla u|^{2}+\frac{s^{2}-\sigma}{\left\vert
y\right\vert ^{2}}|u|^{2}\right)  =\Vert u\Vert_{s\mathbb{A},\sigma}^{2}%
\end{align*}
for all $u\in D_{s\mathbb{A},0}^{1,2}(\Omega,\mathbb{C})^{G}$ such that
$|u|_{2^{\ast}}=1$. This proves that the right-hand side of (\ref{SG}) is
smaller than or equal to the left-hand side. Since $D_{s\mathbb{A},0}%
^{1,2}(\Omega,\mathbb{R})^{G}\subset D_{s\mathbb{A},0}^{1,2}(\Omega
,\mathbb{C})^{G}$ the opposite inequality is obvious. This proves equality
(\ref{SG}). 

If $u$ is a minimizer for the left-hand side of
\eqref{SG}, then the diamagnetic inequality and equality (\ref{SG}) imply that
$|u|$ is a minimizer for the right-hand side.
\end{proof}

\begin{proof}
[Proof of Theorem \ref{thmunbdd}]If $N\geq4$ we take $M=N$ and $G=\mathbb{S}%
^{1}\times O(N-2),$ and if $N=3$ we take $M=2$ and $G=\mathbb{S}^{1}.$ In both
cases we take $\tau\equiv1.$ Then Theorem \ref{existence}, together with Lemma
\ref{AB1}, gives a real-valued nonnegative minimizer $v_{s}$ for (\ref{SG})
such that $|v_{s}|_{2^{\ast}}=(S_{s\mathbb{A},\sigma}^{G})^{1/(2^{\ast}-2)}$.
It follows from the principle of symmetric criticality \cite{p, wi} that
$v_{s}$ satisfies
\[
-\Delta u+\frac{s^{2}-\sigma}{\left\vert y\right\vert ^{2}}\,u=u^{2^{\ast}%
-1}\quad\text{in }\Omega.
\]
Since $\sigma<\kappa_{s}=:\min_{k\in\mathbb{Z}}|k+s|^{2}\leq s^{2},$ standard
regularity results and the maximum principle imply that $v_{s}$ is continuous
and $v_{s}>0$ in $\Omega$. If $N=3$ we apply the moving plane method
\cite{gnn,wi} to show that $u_{n}(y,-z)=u_{n}(y,z)$ after a translation along
the $z$-axis. Therefore $v_{s}$ is $\left[  \mathbb{S}^{1}\times
O(N-2)\right]  $-invariant for every $N\geq3.$ Set
\[
u_{n}:=e^{in\theta}v_{s+n},\quad n\in\mathbb{Z},
\]
where as before, $\theta(y,z)$ is the polar angle of $y$. A simple computation
shows that
\[
\nabla_{s\mathbb{A}}u_{n}=e^{in\theta}\nabla_{(s+n)\mathbb{A}}v_{s+n}.
\]
From this, it is easy to see that $u_{n}$ is a solution to $(\wp_{s}),$ and
the statements (a), (c) and (d) hold. If $v_{t}=v_{s}=v$, then
\[
-\Delta v+\frac{s^{2}-\sigma}{\left\vert y\right\vert ^{2}}\,v=v^{2^{\ast}%
-1}=-\Delta v+\frac{t^{2}-\sigma}{\left\vert y\right\vert ^{2}}\,v,
\]
hence $s^{2}=t^{2}$. So $v_{s}\neq v_{t}$ whenever $s^{2}\neq t^{2}$ which
implies (b). Finally, if $t^{2}>s^{2}$, then
\[
\Vert v_{t}\Vert_{t\mathbb{A},\sigma}^{2}=\Vert v_{t}\Vert_{s\mathbb{A}%
,\sigma}^{2}+(t^{2}-s^{2})\int_{\Omega}\frac{v_{t}^{2}}{\left\vert
y\right\vert ^{2}}>\Vert v_{t}\Vert_{s\mathbb{A},\sigma}^{2}%
\]
and therefore $S_{t\mathbb{A},\sigma}^{G}>S_{s\mathbb{A},\sigma}^{G}$. Since
$\Vert u_{n}\Vert_{s\mathbb{A},\sigma}=\Vert v_{(s+n)}\Vert_{(s+n)\mathbb{A}%
,\sigma}$, it follows that (e) is satisfied.
\end{proof}

\begin{remark}
\emph{ Theorem \ref{thmunbdd} holds for $2<p<2^{\ast}$ if one replaces
$(-i\nabla+s\mathbb{A})^{2}u$ by $(-i\nabla+s\mathbb{A})^{2}u+u$. Also
Theorems \ref{thmpoint} and \ref{existence}, with $(-i\nabla+\mathcal{A}%
)^{2}u$ replaced by $(-i\nabla+\mathcal{A})^{2}u+u$, hold for such $p$. The
proof that the infimum for an appropriate quotient is attained is similar to
that of Proposition 2.3 in \cite{cs}. }
\end{remark}

\section{Equivariant potentials}

\label{sec:rad}Before proving Theorem \ref{rad} we show how symmetry
properties of $\mathcal{A}$ and curl$\mathcal{A}$ are related.

\begin{proposition}
\label{Sing1} Let $G$ be a closed subgroup of $SO(N)$, $\Omega$ a
$G$-invariant domain in $\mathbb{R}^{N}$ and $\mathcal{A}\in L_{loc}%
^{2}(\Omega,\mathbb{R}^{N})$, $N\geq2$.

\begin{enumerate}
\item[(i)] If $\mathcal{A}(gx)=g\mathcal{A}(x)$ for $g\in G,$ $x\in\Omega,$
then $\mathcal{B}(gx)=g\mathcal{B}(x)g^{-1}$.

\item[(ii)] If $\mathcal{B}:=$ \emph{curl}$\mathcal{A}$ (in the sense of
distributions) for some $\mathcal{A}\in L_{loc}^{2}(\Omega,\mathbb{R}^{N})$
and $\mathcal{B}(gx)=g\mathcal{B}(x)g^{-1}$ for all $g\in G,$ $x\in\Omega,$
then there exists $\widehat{\mathcal{A}}\in L_{loc}^{2}(\Omega,\mathbb{R}%
^{N})$ such that \emph{curl}$\widehat{\mathcal{A}}=\mathcal{B}(x)$ and
$\widehat{\mathcal{A}}(gx)=g\widehat{\mathcal{A}}(x).$
\end{enumerate}
\end{proposition}

\begin{proof}
(i) If $\mathcal{A}(gx)=g\mathcal{A}(x)$ then, taking derivatives, we obtain
$(D\mathcal{A})(gx)=g(D\mathcal{A})(x)g^{-1}.$ Hence $(D\mathcal{A}%
)^{t}(gx)=g(D\mathcal{A})^{t}(x)g^{-1}$ and
\[
\mathcal{B}(gx)=(D\mathcal{A})^{t}(gx)-(D\mathcal{A})(gx)=g((D\mathcal{A}%
)^{t}(x)-(D\mathcal{A})(x))g^{-1}=g\mathcal{B}(x)g^{-1}.
\]
(ii) If $\mathcal{B}(gx)=g\mathcal{B}(x)g^{-1}$ for all $g\in G,$ $x\in
\Omega,$ define%
\[
\widehat{\mathcal{A}}(x):=\frac{1}{\left\vert G\right\vert }\int_{G}%
g^{-1}\mathcal{A}(gx)\,dg,
\]
where $dg$ is the Haar measure and $\left\vert G\right\vert :=\int_{G}dg$.
Then, for every $g_{0}\in G,$ $x\in\Omega,$ we have%
\begin{align*}
\widehat{\mathcal{A}}(g_{0}x)  &  =\frac{1}{\left\vert G\right\vert }\int
_{G}g_{0}g_{0}^{-1}g^{-1}A(gg_{0}x)\,dg\\
&  =g_{0}\left(  \frac{1}{\left\vert G\right\vert }\int_{G}g^{-1}%
\mathcal{A}(gx)\,dg\right)  =g_{0}\widehat{\mathcal{A}}(x).
\end{align*}
Moreover, since%
\[
(D\widehat{\mathcal{A}})(x)=\frac{1}{\left\vert G\right\vert }\int_{G}%
g^{-1}(D\mathcal{A})(gx)g\,dg\quad\text{and}\quad(D\widehat{\mathcal{A}}%
)^{t}(x)=\frac{1}{\left\vert G\right\vert }\int_{G}g^{-1}(DA)^{t}(gx)g\,dg,
\]
we have that%
\[
\text{curl}\widehat{\mathcal{A}}(x)=\frac{1}{\left\vert G\right\vert }\int
_{G}g^{-1}\mathcal{B}(gx)g\,dg=\frac{1}{\left\vert G\right\vert }\int
_{G}\mathcal{B}(x)\,dg=\mathcal{B}(x),
\]
as claimed.
\end{proof}

\begin{proof}
[Proof of Theorem \ref{rad}](i) Write $\mathcal{A}(x)=\mathcal{A}_{\tau
}(x)+\nu(x)x$ with $\mathcal{A}_{\tau}(x)\cdot x=0,$ $\nu(x)\in\mathbb{R}.$
Assume that $\mathcal{A}_{\tau}(x)\neq0$ for some $x\in\Omega,$ $x\neq0.$
Since $N\geq3,$ we may choose $\gamma\in SO(N)$ such that $\gamma x=x$ and
$\gamma\mathcal{A}_{\tau}(x)\neq\mathcal{A}_{\tau}(x)$ (just take $\gamma$ to
be an appropriate rotation around the line generated by $x$). Then,%
\[
\gamma\mathcal{A}_{\tau}(x)+\nu(x)x=\gamma\mathcal{A}(x)=\mathcal{A}(\gamma
x)=\mathcal{A}(x)=\mathcal{A}_{\tau}(x)+\nu(x)x
\]
and, therefore, $\gamma\mathcal{A}_{\tau}(x)=\mathcal{A}_{\tau}(x).$ This is a
contradiction. It follows that $\mathcal{A}(x)=\nu(x)x$ and, since
\[
\nu(gx)gx=\mathcal{A}(gx)=g\mathcal{A}(x)=\nu(x)gx\text{ \ \ \ for all }g\in
SO(N),
\]
we have that $\nu(\left\vert x\right\vert )=\nu(x),$ as claimed. Consequently,
$\mathcal{A}(x)=\nabla\varphi(x)$ where $\varphi(x)=\psi(|x|)$ and $\psi(r)$
is a primitive for $\nu(r)r$.

(ii) This is an immediate consequence of (i) and Proposition \ref{Sing1}.
\end{proof}

As we have pointed out in the introduction, it follows from Theorem \ref{rad}
that $SO(N)$-equivariant vector potentials can be gauged away if $N\geq3$ and
there is a one-to-one correspondence between radially symmetric solutions to
(\ref{i1}) and to a similar non-magnetic problem.

Below we shall show by a simple example that radially symmetric solutions to
\eqref{i1} may exist also when $G$ is a proper subgroup of $SO(N)$.

\begin{example}
\emph{ Let $\Omega=\mathbb{R}^{4}\smallsetminus\{0\}$ and
\[
\mathcal{A}(x):=\frac{1}{|x|^{2}}(-x_{2},x_{1},-x_{4},x_{3}).
\]
Then $\mathcal{A}$ is $SO(2)\times SO(2)$ but not $SO(4)$-invariant and, since
$\text{curl}\mathcal{A} \ne0$, $\mathcal{A}$ cannot be gauged away. Note that,
if $u=u(|x|),$ then $\nabla u(x)\cdot\mathcal{A}(x)=0$ and, therefore,
\[
(-i\nabla+\mathcal{A})^{2}u=-\Delta u+\left\vert \mathcal{A}\right\vert
^{2}u=-\Delta u+\frac{1}{\left\vert x\right\vert ^{2}}u,
\]
cf. (\ref{gradradial}). So $u$ is a radial solution to problem (\ref{point})
with $\sigma=0$ iff $u$ solves
\begin{equation}
-\Delta u+\frac{1}{\left\vert x\right\vert ^{2}}u=|u|^{2^{\ast}-2}%
u,\text{\qquad}u\in D^{1,2}(\mathbb{R}^{4},\mathbb{C})^{SO(4)}.\label{nonmag}%
\end{equation}
Theorem \ref{thmpoint} asserts that a solution to (\ref{nonmag}) exists. }
\end{example}

\end{document}